\documentclass[twoside]{amsart}
\usepackage{mathrsfs}
\usepackage{amsfonts}
\usepackage{amssymb}
\usepackage{amsmath}
\usepackage{amsthm}
\usepackage[all]{xy}
\usepackage{mathdots}

\usepackage[verbose,colorlinks,pagebackref]{hyperref}
\hypersetup{
  colorlinks   = true, 
  urlcolor     = blue, 
  linkcolor    = blue, 
  citecolor   = green 
}

\usepackage[alphabetic,nobysame]{amsrefs}
\usepackage{lastpage}

\RequirePackage{amsmath} \RequirePackage{amssymb}
\usepackage{amscd,latexsym,amsthm,amsfonts,amssymb,amsmath,amsxtra}

\newcommand{\bpm}{\begin{pmatrix}}
\newcommand{\epm}{\end{pmatrix}}
\newcommand{\bsm}{\begin{smallmatrix}}
\newcommand{\esm}{\end{smallmatrix}}
\newcommand{\bspm}{\left(\begin{smallmatrix}}
\newcommand{\espm}{\end{smallmatrix}\right)}
\newcommand{\bbm}{\begin{bmatrix}}
\newcommand{\ebm}{\end{bmatrix}}

    \newcommand{\BA}{{\mathbb {A}}} 
    \newcommand{\C}{{\mathbb {C}}}

     \newcommand{\CB}{{\mathcal {B}}}

    \newcommand{\CM}{{\mathcal {M}}} 
    \newcommand{\CO}{{\mathcal {O}}} 
     
    \newcommand{\CS}{{\mathcal {S}}} 
     
    \newcommand{\CW}{{\mathcal {W}}}

     \newcommand{\fo}{{\mathfrak{o}}}  \newcommand{\fp}{{\mathfrak{p}}}

    \newcommand{\RG}{{\mathrm {G}}}

    \newcommand{\RO}{{\mathrm {O}}}

    \newcommand{\RU}{{\mathrm {U}}}

    \newcommand{\bW}{{\mathbf {W}}}

    \newcommand{\lenth}{{\mathrm {\lenth}}}

     \newcommand{\GL}{{\mathrm{GL}}}
    \newcommand{\Hom}{{\mathrm{Hom}}} 
    \newcommand{\Ind}{{\mathrm{Ind}}} \newcommand{\ind}{{\mathrm{ind}}}

    \newcommand{\id}{{\mathrm{id}}}

    \newcommand{\cond}{\mathrm{cond}} 
    \renewcommand{\Re}{{\mathrm{Re}}} 
    \newcommand{\Res}{{\mathrm{Res}}}

    \newcommand{\Mat}{{\mathrm {Mat}}}

 \newcommand{\Vol}{{\mathrm{Vol}}}

\newcommand{\supp}{\mathrm{supp}}

 \newcommand{\SL}{{\mathrm{SL}}}
 \newcommand{\SO}{{\mathrm{SO}}}
 \newcommand{\GSO}{{\mathrm{GSO}}}

\newcommand{\vol}{{\mathrm{vol}}}

 \newcommand{\Sp}{{\mathrm{Sp}}}

    \newcommand{\bx}{{\bf {x}}}   \newcommand{\bm}{{\bf {m}}}   \newcommand{\bn}{{\bf {n}}} \newcommand{\bt}{{\bf {t}}}

\newcommand{\diag}{{\mathrm {diag}}}

    \newcommand{\wt}{\widetilde}

    \newcommand{\wpair}[1]{\left\{{#1}\right\}}

    \newcommand{\ov}{\overline}

     \newcommand{\ra}{\rightarrow}

    \theoremstyle{plain}

       \newtheorem*{theorem*}{Theorem}

    \newtheorem{thm}{Theorem}[section] \newtheorem{cor}[thm]{Corollary}
    \newtheorem{lem}[thm]{Lemma}  \newtheorem{prop}[thm]{Proposition}

    \numberwithin{equation}{section}

\usepackage[top=1in, bottom=1in, left=1.25in, right=1.25in]{geometry}

\title{On a refined local converse theorem for $\SO(4)$}

\author{Pan Yan}
\address{Department of Mathematics, University of Arizona, Tucson, AZ 85721 USA}
\email{panyan@math.arizona.edu}

\author{Qing Zhang}
\address{School of Mathematics and Statistics, Huazhong University of Science and
Technology, Wuhan, 430074, China}
\thanks{The second named author was partially supported by NSFC grant 12371010.}
\email{qingzh@hust.edu.cn}


\begin{document}
\maketitle
\begin{abstract}
Recently, Hazeltine-Liu, and independently Haan-Kim-Kwon, proved a local converse theorem for $\mathrm{SO}_{2n}(F)$ over a $p$-adic field $F$, which says that, up to an outer automorphism of $\mathrm{SO}_{2n}(F)$, an irreducible generic representation of $\mathrm{SO}_{2n}(F)$ is uniquely determined by its twisted gamma factors by generic representations of $\mathrm{GL}_k(F)$ for $k=1,\dots,n$. It is desirable to remove the ``up to an outer automorphism" part in the above theorem using more twisted gamma factors, but this seems a hard problem.  In this paper, we provide a solution to this problem for the group $\mathrm{SO}_4(F)$, namely,  we show that a generic supercuspidal representation $\pi$ of $\mathrm{SO}_4(F)$ is uniquely determined by its $\mathrm{GL}_1$, $\mathrm{GL}_2$ twisted local gamma factors and a twisted exterior square local gamma factor of $\pi$. 
\end{abstract}

\section{Introduction}

Let $F$ be a local field. Let $G_n$ denote the split classical groups $\SO_{2n+1}, \Sp_{2n}$ or quasi-split classical groups $\RU_{2n}, \RU_{2n+1}$ over $F$. Here in the quasi-split case, the unitary groups are defined with respect to a fixed quadratic extension $E/F$.  Given an irreducible generic representation $\pi$ of $G_n(F)$ and an irreducible generic representation $\tau$ of $\GL_n(F)$ (in the split case) or $\GL_n(E)$ (in the unitary group case), one can associate a local gamma factor $\gamma(s,\pi\times \tau,\psi)$ using either Langlands-Shahidi method, or appropriate Rankin-Selberg integrals; see \cite{Kaplan} for a nice survey on the definition of these gamma factors using Rankin-Selberg method (and see \cites{BAZ, Morimoto-gamma, CW} in the unitary group case). Here $\psi$ is a fixed additive character of $F$. The local converse theorems for these groups, see \cite{JS, Chai, JL, Morimoto, Sp(2r), U(2r+1), Jo, YZ} for various cases, assert that the representation $\pi$ is uniquely determined by the family $\gamma(s,\pi\times \tau, \psi)$ as $\tau$ runs over all irreducible generic representations of $\GL_m(F)$ (or $\GL_m(E)$) for all $m$ with $1\le m\le n$. See \cite{LZ} for more references. However, a similar converse theorem for $\SO_{2n}$ is false in general because of the existence of an outer automorphism, which is defined by conjugation of an element $c\in \RO_{2n}(F)-\SO_{2n}(F)$ (for an explicit matrix form of $c$, see Section~\ref{subsection-HL-Theorem}). In fact, the twisted gamma factors of $\SO_{2n}(F)\times \GL_k(F)$ cannot distinguish $\pi$ and $c\cdot \pi$ \cite{HL: p-adic, HL: finite, HKK}. Here $c\cdot \pi$ is the representation of $\SO_{2n}(F)$ defined by $(c\cdot\pi)(g)=\pi(cgc^{-1})$. Thus if one only uses the twisted gamma factors of $\SO_{2n}(F)\times \GL_k(F)$, the best possible result one could expect is that these gamma factors can determine $\pi$ up to the conjugation by $c$. This is actually the main result of \cite{HL: p-adic} and \cite{HKK}. A similar phenomenon is also known over finite fields, see \cite{HL: finite}.  This phenomenon is consistent with the work of Arthur on the local Langlands correspondence \cite{Arthur}.

Let $\pi$ be an irreducible generic representation of $\SO_{2n}(F)$. Since $\pi$ and $c\cdot \pi$ are generic with respect to the same Whittaker datum, they cannot be in the same $L$-packet unless $\pi\cong c\cdot \pi$ by the uniqueness of generic element (with respect to a fixed Whittaker datum) in an $L$-packet as proved in \cite{Atobe}.  It is thus desirable to find additional invariants to distinguish $\pi$ and $c\cdot \pi$ if $\pi\ncong c\cdot \pi.$ 

In this paper, for the small group $\SO_4(F)$, we construct an additional twisted local gamma factor $\gamma(s,\pi,\wedge^2_+\times \eta, \psi)$, where $\wedge^2_+$ is a constituent of $\wedge^2$ of the dual group $\SO_{4}(\C)$ (see explanation below), and $\eta$ is a character of $F^\times$. We show that, the twisted gamma factors of $\SO_4\times \GL_k$ for $k=1,2$, plus the gamma factors $\gamma(s,\pi,\wedge^2_+\times \eta, \psi)$ can determine the representation $\pi$ uniquely. Namely, we show the following
\begin{thm}\label{converse theorem}
Let $F$ be a $p$-adic field. Let $\pi,\pi'$ be two irreducible $\psi$-generic supercuspidal representations of $\SO_4(F)$ with the same central character. If $\gamma(s,\pi\times \tau,\psi)=\gamma(s,\pi'\times \tau,\psi)$ and $\gamma(s,\pi,\wedge^2_+\times \eta,\psi)=\gamma(s,\pi',\wedge^2_+\times \eta,\psi)$ for any irreducible generic representation $\tau$ of $\GL_k(F)$ with $k=1,2$ and any quasi-character $\eta$ of $F^\times$, then $\pi\cong \pi'.$
\end{thm}

Let us now explain the representation $\wedge^2_+$ of $\SO_4(\C)$. In fact, such a representation can be defined for $\SO_{2n}(\C)$ where $n$ is any positive integer, as explained below. Let $W=\C^{2n}$ and let $\{e_i, 1\le i\le 2n\}$ be the standard basis of $W$. Consider the bilinear form $Q$ on $W$ defined by 
\begin{equation*}
Q(e_i, e_j)=
\begin{cases}
1,  &i+j= 2n+1, \\
0, &i+j\ne 2n+1.
\end{cases}
\end{equation*}
In other words, the bilinear form $Q$ is given by $Q(v_1,v_2)=v_1J_{2n}v_2^t$, where $v_i\in W$ is viewed as a row vector, and $J_{2n}=\left(\begin{smallmatrix}&&1\\ &\iddots &\\ 1&& \end{smallmatrix}\right)\in \GL_{2n}$. Using the bilinear form $Q$, we can realize the groups $\RO_{2n}(\C)$ and $\SO_{2n}(\C)$ as
\begin{equation*}
\RO_{2n}(\C)=\wpair{g\in \GL_{2n}(\C): Q(gw_1, gw_2)=Q(w_1,w_2) \text{ for all }   w_1, w_2\in W},
\end{equation*}
and
\begin{equation*}
\SO_{2n}(\C)=\wpair{g\in \RO_{2n}(\C): \det(g)=1}.
\end{equation*}
We fix an isomorphism $\wedge^{2n}(W)\to \C$ given by 
\begin{equation*}
e_1\wedge e_2\wedge \dots \wedge e_{2n}\mapsto 1.
\end{equation*}
Then the bilinear form $Q$ defines an isomorphism 
\begin{equation*}
\varphi_1: \wedge^n(W)\to \wedge^n(W^*)
\end{equation*}
and the wedge product $\wedge^n(W)\times \wedge^n (W)\to \wedge^{2n}(W) \cong \C$ determines an isomorphism 
\begin{equation*}
\varphi_2: \wedge^n(W^*)\to \wedge^n(W).
\end{equation*} 
Denote $\rho=\varphi_2\circ \varphi_1$. Then we can check that $\rho^2=\id$. Let $W_j\subset \wedge^n(W)$ be the $j$-eigenspace of $\rho$ for $j\in \wpair{\pm 1}$. 
Each $W_j$ is invariant under the natural action of $\SO_{2n}(\C)$. Denote the corresponding action of $\SO_{2n}(\C)$ on $W_{j}$ by $\wedge^n_{+}$ (resp. $\wedge^n_-$) if $j=1$ (reps. $j=-1$). Then we have 
$$\wedge^{n}=\wedge^n_+\oplus \wedge^n_-.$$
In fact, by \cite{FH}*{Theorem 19.2}, both $\wedge^n_+$ and $\wedge^n_-$ are irreducible. Actually, one can also check that $\wedge^n_{-}=c\cdot \wedge^n_{+}$, where $c$ denotes the outer automorphism of $\SO_{2n}(\C)$. Note that, the definition of $\wedge^n_+$ depends on the choice of the isomorphism $\wedge^{2n}(W)\to \C $. Thus there is no canonical choice of $\wedge_+^n.$ 

We assume that $F$ is a $p$-adic field and $W_F'$ is the Weil-Deligne group of $F$. Let $\pi$ be an irreducible representation $\SO_{2n}(F)$. Arthur \cite{Arthur} defined  the local Langlands parameter $\phi_\pi: W_F'\to \SO_{2n}(\C)$ of $\pi$ up to the outer conjugation $c\in \RO_{2n}(\C)-\SO_{2n}(\C)$.  Let $\sigma$ be an irreducible representation of $\GL_k(F)$ for some positive integer $k$ then we write $\gamma(s,\pi,\wedge^n_{\pm}\times \sigma,\psi):=\gamma(s,(\wedge^n_{\pm}\circ \phi_\pi)\otimes \phi_\sigma,\psi)$ by abuse of notation, where $\psi$ is a fixed nontrivial additive character of $F$, and $\phi_\sigma: W_F'\to \GL_k(\C)$ is the local Langlands parameter of $\sigma$. Since there is no canonical choice of $\phi_\pi$ because of the existence of the outer automorphism $c$, the above notation of $ \gamma(s,\pi,\wedge^n_{\pm}\times \sigma,\psi)$ is vague. Thus it is desirable to give pure representation theoretic definition of these local gamma factors $\gamma(s,\pi,\wedge^n_+\times \sigma, \psi)$.

In this article, for the small group $\SO_4$ case, given an automorphic cuspidal generic representation $\pi$ of $\SO_4(\BA)$ where $\BA$ is the ring of adeles of a global field $F$, and a character $\eta$ of $F^\times \backslash \BA^\times$,  we construct a global integral which is Eulerian and at an unramified place, it represents the local $L$-function $L(s,\pi_v, \wedge^2_+\times \eta_v)$. The local gamma factors  $\gamma(s,\pi_v, \wedge^2_+\times \eta_v, \psi_v)$ are then constructed using the local functional equations of those local integrals. It turns out that the new local gamma factor can determine the values of certain Whittaker functions of $\pi$ on the Bruhat cells, at which only the sum of Whittaker functions of $\pi$ and $c\cdot \pi$ can be determined using the $\GL_k$-twisted gamma factors, as shown in \cite{HL: p-adic, HKK}. Theorem \ref{converse theorem} follows from this consideration and the results of \cite{HL: p-adic, HKK} easily.

Although we can only handle the small rank case, we expect that our result can shed some light on the general problem: find enough gamma factors for generic representations of $\SO_{2n}(F)$ so that they can distinguish a representation $\pi$ of $\SO_{2n}(F)$ and its $c$-conjugate $c\cdot \pi$ if $\pi\ncong c\cdot \pi$. After  Theorem \ref{converse theorem}, one might expect the twisted gamma factor $\gamma(s,\pi,\wedge^n_+\times \sigma,\psi)$ will be enough for the purpose, where $\pi$ is an irreducible generic representation of $\SO_{2n}(F)$ and $\sigma$ is an irreducible generic representation of $\GL_k(F)$. In other words, if $\pi$ is an irreducible generic representation of $\SO_{2n}(F)$ such that $\pi \ncong c\cdot\pi$, then one should expect that there is an irreducible generic representation $\sigma$ of $\GL_k(F)$ for some $k\le n$ such that $\gamma(s,\pi, \wedge^n_+\times \sigma,\psi)\ne \gamma(s,c\cdot\pi,\wedge^n_+\times\sigma,\psi).$ But one referee told us that they could find a counterexample even for $n=3$. Thus one needs other gamma factors to distinguish $\pi$ and $c\cdot \pi$.

The paper is organized as follows. In Section~\ref{section-local}, we introduce the local zeta integral and compute it at unramified places, and use it to define the gamma factor $\gamma(s,\pi,\wedge^2_+\times \eta, \psi)$.  We also construct a global integral and show that it is Eulerian and its local piece is the local zeta integral we study in this paper. In Section~\ref{section-partial-bessel-functions} we review the theory of partial Bessel functions as well as a local converse theorem of Hazeltine and Liu. Finally, in Section~\ref{section-proof}, we prove Theorem~\ref{converse theorem}.

\section*{Acknowledgement}
The authors would like to thank Alex Hazeltine and Baiying Liu for sending us their manuscripts \cite{HL: finite, HL: p-adic} and for helpful communications. The idea of this paper was initiated when the second named author was a graduate student at Ohio State University. We thank our advisor Jim Cogdell for his constant encouragement and support. We thank Chi-Heng Lo for helpful communications. We greatly appreciate the anonymous referees for their careful reading and very helpful suggestions. In particular, we would like to thank one referee for informing us that they have a counterexample to show that the gamma factors $\gamma(s,\pi, \wedge^n_+\times \sigma, \psi)$ are not enough to distinguish $\pi$ and $c\cdot \pi$ even for $n=3$.

\section*{Notations}
Let $F$ be a $p$-adic field. For an algebraic group $\mathbf{G}$ we denote its group of $F$-points by $G(F)$ or simply by $G$. For a positive integer $r,$ let $J_r=\left(\begin{smallmatrix}&&1\\ &\iddots &\\ 1&& \end{smallmatrix}\right)\in \GL_r(F)$. We realize the special orthogonal group $\SO_r$ as the group of all $r\times r$ matrices which satisfy ${}^t\! g J_r g= J_r$. In this paper we will focus on the case when $r=4$. 

Let $B$ denote the upper triangular Borel subgroup of $\SO_4$. Then we have the Levi decomposition $B=T U$ where $T$ is the torus and $U$ is the unipotent subgroup.  More specifically, we have
 $$T=\wpair{t(a_1,a_2) :=\diag(a_1,a_2,a_2^{-1},a_1^{-1}), a_1, a_2\in F^\times}.$$
 Consider the roots $\alpha,\beta$ of $\SO_4$ defined by 
$$\alpha(t(a_1,a_2))=a_1/a_2, \beta(t(a_1,a_2))=a_1a_2.$$
Then the set $\Delta=\wpair{\alpha,\beta}$  is the set of simple roots of $\SO_4$. For a root $\gamma,$ let $U_\gamma$ be the one parameter subgroup and we fix an isomorphism $\bx_\gamma: F\to U_\gamma$.
 Let
$$ U=\wpair{u(x,y):=\bx_{\alpha}(x)\bx_{\beta}(y)=\begin{pmatrix} 1& x&&\\ &1&&\\ &&1&-x\\ &&&1 \end{pmatrix}\begin{pmatrix} 1&& y &\\ &1&&-y \\ &&1&\\ &&&1 \end{pmatrix},x,y\in F}.$$
Let $\ov U$ be the opposite of $U$, and denote $\bar u(x,y):=\bx_{-\alpha}(x)\bx_{-\beta}(y),$ for $x,y\in F$.

Let $P=MN$ be the Siegel parabolic subgroup of $\SO_4$, with the Levi subgroup $M\cong \GL_2$ and unipotent radical $N=\wpair{\bx_\beta(y), y\in F}$. For $h\in \GL_2$, we let $h^*=J_2 {}^t\! h^{-1}J_2$ and denote 
$$\bm(h):=\begin{pmatrix}h&\\ & h^* \end{pmatrix}\in M.$$

Let $\bW=\wpair{1,s_{\alpha}, s_\beta, s_\alpha s_\beta}$ be the Weyl group of $\SO_4$, where $s_\alpha$ (resp. $s_\beta$) is the simple reflection defined by $\alpha$ (resp. $\beta$). 

Let $\wpair{e_1,e_2,e_{3},e_{4}}$ be the standard basis of the dimension 4 quadratic space $(W=\C^4,Q)$ which defines $\SO_4(\C)$ with matrix $\wpair{Q(e_i,e_j)}=J_4$. The subspace $W_1\subset \wedge^2(W)$ as in the introduction is spanned by $\CB=\wpair{ e_1\wedge e_3, e_1\wedge e_4+e_3\wedge e_2, e_2\wedge e_4}$. Let $\wedge^2_+$ be the action of $\SO_4(\C)$ on $W_1$. For a torus element $\diag(a,b,b^{-1},a^{-1})\in \SO_4(\C)$, we have
$$\wedge^2_+(\diag(a,b,b^{-1},a^{-1}))=\diag(ab^{-1},1,a^{-1}b),$$
with respect to the ordered basis $\CB$ of $W_1$.

Let $\pi$ be an irreducible representation of $\SO_4(F)$ over $F$. Let $\varpi$ be a uniformizer of $F$ and let $q$ be the cardinality of the residue field of $F$. We can consider the $L$-function $L(s,\pi,\wedge^2_+)$. In particular, if $\pi$ is unramified with Satake parameter $\diag(a,b,b^{-1},a^{-1})\in \SO_4(\mathbb{C})$, we have
$$L(s,\pi, \wedge^2_+)=\frac{1}{(1-q^{-s} )(1-ab^{-1}q^{-s})(1-a^{-1}bq^{-s})}.$$
If in addition $\eta$ is a character of $F^\times$, we then have 
$$L(s,\pi, \wedge^2_+\times \eta)=\frac{1}{(1-\eta(\varpi)q^{-s} )(1-ab^{-1}\eta(\varpi)q^{-s})(1-a^{-1}b\eta(\varpi)q^{-s})}.$$
Consider $c= \diag(1,J_2,1)\in \RO_4(F)-\SO_4(F).$ It is used to define the outer automorphism on $\SO_4$.

In the group $\SL_2(F)$, we use the following notations. Let $T_2$ be the torus of $\SL_2(F)$, which consists of elements of the form $t_{\SL_2}(a)=\diag(a,a^{-1})$ where $a\in F^\times$. Let $N_2$ be the upper triangular subgroup of $\SL_2(F)$, which consists elements of the form $\bn(x)=\left(\begin{smallmatrix} 1 &x\\ &1 \end{smallmatrix}\right)$. Let $B_2=T_2N_2$ be the upper triangular Borel subgroup of $\SL_2(F)$. Let $w_2=\left(\begin{smallmatrix} &1\\ -1&\end{smallmatrix} \right)$ denote the nontrivial Weyl element of $\SL_2(F)$.

\section{The local zeta integral and gamma factor for $\wedge^2_+$}\label{section-local}
Let $F$ be a non-archimedean local field and $\psi$ be a fixed additive character of $F$. Let $\fo_F$ be the ring of integers of $F$, $\fp$ the maximal ideal of $\fo_F$, $\varpi$ be a fixed uniformizer of $F$ and $q$ the cardinality of $\fo_F/\fp$. Let $(\cdot, \cdot)_F$ be the local Hilbert symbol.  In this section, given a generic irreducible unramified representation $\pi$ of $\SO_4(F)$ and a character $\eta$ of $F^\times$, we construct a local zeta integral which represents the $L$-function $L(s,\pi,\wedge^2_+\times \eta)$. The integral is pretty much like the integral of $\SL_2\times \GL_1$ constructed in \cite{GPSR}, which is not surprising considering that, on the Galois side, $\wedge^2_+$ defines a representation $W_F'\to \SO_3(\C)$, and hence corresponds to a representation of $\SL_2(F)$ through the local Langlands correspondence.   
\subsection{Weil representations of $\widetilde \SL_2$}\label{subsection: Weil}
We recall the theory of the Weil representation of $\wt \SL_2(F)$, the metaplectic double cover of $\SL_2(F)$. Recall that the product on $\widetilde \SL_2(F)$ is given by
$$(g_1,\zeta_1)(g_2,\zeta_2)=(g_1 g_2, \zeta_1 \zeta_2 \mathbf{c}(g_1,g_2)),$$
where $\mathbf{c}:\SL_2(F)\times \SL_2(F)\ra \wpair{\pm1}$ is defined by
$$\mathbf{c}(g_1,g_2)=(\bx(g_1),\bx(g_2))_F (-\bx(g_1)\bx(g_2), \bx(g_1g_2))_F,$$
where $$\bx\begin{pmatrix} a& b\\ c& d \end{pmatrix}=\left\{\begin{array}{lll} c, & c\ne 0, \\ d, & c=0. \end{array} \right.$$

For a subgroup $A$ of $\SL_2(F)$, we denote by $\widetilde A$ the preimage of $A$ in $\widetilde \SL_2(F)$, which is a subgroup of $\widetilde \SL_2(F)$. For an element $g\in \SL_2(F)$, we sometimes also let $g$ denote the element $(g,1)$ inside $\widetilde \SL_2(F)$. 

A representation $\pi$ of $\widetilde \SL_2(F)$ is called genuine if $\pi(\zeta g)=\zeta \pi(g)$ for all $g\in \widetilde \SL_2(F)$ and $\zeta \in \mu_2$. Let $\omega_{\psi}$ be the Weil representation of $\widetilde \SL_2(F)$ realized on $\CS(F)$, the space of Bruhat-Schwartz functions on $F$. For any $f\in \CS(F)$, the action is given by the following formulas: 
\begin{align*}
\left(\omega_{\psi} (w_2) f\right)(x)&=\gamma(\psi)\hat f(x),\\
\left(\omega_{\psi}(\bn(b))f\right) (x)&=\psi(bx^2)f(x), b\in F,\\
\left(\omega_{\psi}(t_{\SL_2}(a))f\right) (x)&=|a|^{1/2}{\mu_\psi(a)}f(ax), a\in F^\times,
\end{align*}
and 
$$ \omega_{\psi}(\zeta)f(x)=\zeta f(x), \zeta \in \mu_2. $$
Several notations need to be explained. Here $\hat f(x)=\int_F f(y)\psi(2xy)dy$ is the Fourier transform of $f(x)$ and the Haar measure $dy$ is normalized so that $(\hat f)^{\hat~}(x)=f(-x)$. The function $\mu_\psi$ is defined as $\mu_\psi(a)=\frac{\gamma(\psi)}{\gamma(\psi_a)}$, where the constant $\gamma(\psi)$ is the Weil index associated to the character $\psi$. For $a\in F^\times$, the character $\psi_a$ is defined by $\psi_a(x)=\psi(ax)$. 

The product in $\widetilde T_2$ is given by the Hilbert symbol, i.e., 
$$(t_{\SL_2}(a),1)(t_{\SL_2}(b),1)=(t_{\SL_2}(ab),(a,b)_F).$$
 The function $\mu_\psi$ satisfies
$$\mu_{\psi}(a)\mu_\psi(b)=\mu_\psi(ab)(a,b)_F,$$
and thus defines a genuine character of $\widetilde T_2$.
All of the above facts in this subsection can be found in \cite[Section 1]{GPS80}.

\subsection{The local zeta integral}\label{subsection: local zeta integral} Denote $p: \widetilde \SL_2(F)\ra \SL_2(F)$ the projection. For a character $\eta$ of $F^\times$ and $s\in \C$, consider the character $\eta_s$ of $\widetilde T_2$ defined by $\eta_s((t_{\SL_2}(a),\zeta))=\eta_s(a)=\eta(a)|a|^s$. Then $\eta_s$ factors through $\SL_2(F)$ and hence is not genuine. The product $\mu_\psi \eta_s$ is a genuine character of $\wt T_2.$ Let $P_2=T_2N_2$ be the upper triangular parabolic subgroup of $\SL_2$ with unipotent radical $N_2$. Then $\widetilde P_2=\widetilde T_2\ltimes N_2$. Extend $\mu_\psi\eta_s$ to a character of $\widetilde P_2$ by letting the action of $N_2$ be trivial. We consider the normalized induced representation
$$I(s,\eta, \psi)=\Ind_{\widetilde P_2}^{\widetilde \SL_2}(\mu_\psi \eta_{s-1/2}).$$
A section $f_s\in I(s,\eta,\psi)$ satisfies 
$$f_s( (t_{\SL_2}(a),\zeta) \tilde h)=\zeta \mu_\psi(a) \eta(a)\delta_{\wt P_2}(a)^{1/2}|a|^{s-1/2}f_s(\tilde h)=\zeta \mu_\psi(a) \eta(a)|a|^{s+1/2}f_s(\tilde h),$$
where $\delta_{P_2}$ is the modulus character of $\wt P_2$.
Note that for any $\phi\in \CS(F)$ and $f_s\in I(s,\eta, \psi)$, the function 
$$\omega_{\psi^{-1}}(\tilde h)\phi(1)f_s(\tilde h)$$
on $\widetilde \SL_2(F)$ is trivial on $\mu_2$, and thus defines a function on $\SL_2$.

Let $\psi_{U}$ be the generic character of $U$ defined by 
$$\psi_U(u(x,y))=\psi\left(x-2y\right).$$
The choice of $\psi_U$ defined above is to make the integrals for $\SO_4\times \GL_1$ and $\SO_4\times \GL_2$ easier; see \cite{Kaplan, HL: finite, HL: p-adic}. In fact, if we fix a nontrivial additive character $\psi_0$ of $F$ a priori and take $\psi(x)=\psi_0(\frac{1}{4}x)$, the character $\psi_U$ is the same as those defined in \cite{Kaplan, HL: finite, HL: p-adic}.  Let $(\pi,V_\pi)$ be a $\psi_{U}$-generic representation of $\SO_4(F)$. For $W\in \CW(\pi,\psi_{U})$, $\phi\in \CS(F)$ and $f_s\in I(s,\eta,\psi^{-1})$, we consider the local zeta integral

$$\Psi(W,\phi,f_s)=\int_{N_2\backslash  \SL_2(F)}W\left(\bm(h)\right)(\omega_{\psi^{-1}}(h)\phi)(1)f_s(h)dh.$$

Formally this is well-defined because for any $n\in N_2$, we have $W(\bm(nh))=\psi(n)W(\bm(h))$ and $\omega_{\psi^{-1}}(nh)\phi(1)=\psi^{-1}(n)\omega_{\psi^{-1}}(h)\phi(1)$.

\begin{lem}\label{lem: convergence}
The local zeta integral $\Psi(W,\phi,f_s)$ is absolutely convergent for $\Re(s)\gg 0$ and defines a meromorphic function of $q^{-s}$. Moreover, we can choose $W,\phi,f_s$ such that $\Psi(W,\phi,f_s)$ is a nonzero constant.
\end{lem}
\begin{proof}
The first assertion follows from the asymptotic behavior of $W$ and we omit the details. See \cite{GPSR} for a proof of similar results. The second assertion will be proved by \eqref{eq-value-of-local-integral} (in the proof of Proposition~\ref{prop: main}).  
\end{proof}

\subsection{The local functional equation}

Recall that $N=\wpair{\bx_{\beta}(y): y\in F}$. Denote $\psi_N=\psi_{U}|_{N}$.  For $A\in \GL_2(F)$, we have
$$\bm(A)\bx_{\beta}(y) \bm(A)^{-1}=\bx_{\beta}(\det(A)y).$$
Thus the stabilizer of $\psi_N$ in $M\cong \GL_2(F)$ is $\SL_2(F)$. From this relation, it is easy to check the following

\begin{lem}\label{lem: conjugation}
For any $n\in N$ and $h\in \widetilde \SL_2(F)$, we have
$$W(\bm(p(h))n)=\psi_N(n)W(\bm(p(h))). $$
\end{lem}

We now prove the following results.

\begin{prop}\label{multiplicity one}
Except for a finite number of $q^{-s}$, there is at most one trilinear form $T$ on $\CW(\pi, \psi_{U}) \times \omega_{\psi^{-1}}\times I(s,\eta,\psi^{-1})$ such that
$$T(\pi(n)W, \phi, f_s)=\psi_N(n)T(W, \phi, f_s), \forall n\in N$$
and $$T(\pi(\bm(p(h)))W, \omega_{\psi^{-1}}(h)\phi, r(h)f_s)=T(W, \phi, f_s), \forall h\in \widetilde \SL_2(F).$$
Here $r(h)$ denotes the right translation by $h$.
\end{prop}

\begin{proof}
A trilinear form $T$ as in the proposition defines an element in
\begin{align*} &\Hom_{\widetilde \SL_2}(\pi_{N,\psi_N}\otimes \omega_{\psi^{-1}}\otimes I(s,\eta,\psi^{-1}), \C)\\
=&\Hom_{\widetilde P_2}(\pi_{N,\psi_N}\otimes \omega_{\psi^{-1}}, \widetilde \eta_s^{-1}),
\end{align*}
where $\pi_{N,\psi_N}$ is the Jacquet module of $\pi$ with respect to $(N, \psi_N)$ and it is viewed as a representation of $\widetilde \SL_2$ by the projection $p$, and $\widetilde \eta_s=\mu_{\psi^{-1}}\eta_{s-1/2}$. 

We consider the representation $\pi_{N,\psi_N}$ of $\SL_2$. We have the exact sequence
$$0\ra \oplus_{\alpha \in F^\times/ F^{\times, 2}}\ind_{N_2}^{P_2} ((\pi_{N,\psi_N})_{N_2, \psi_{\alpha}})\ra \pi_{N,\psi_N} \ra (\pi_{N,\psi_N})_{N_2}\ra 0.$$
Here $\ind$ denotes non-normalized compact induction. Recall that $\psi_\alpha$ denotes the character of $F$ defined by $\psi_\alpha(x)=\psi(\alpha x)$. This exact sequence follows from a simple application of the general theory of \cite{BZ}; see \cite{U11}*{(1.1)} for a proof of a similar situation. Since the Jacquet functor preserve admissibility and sends a finitely generated representation to a finitely generated representation, $\pi_{N,\psi_N}$ has finite length as a representation of $\SL_2(F)$ and $(\pi_{N,\psi_N})_{N_2} $ is finite dimensional.  By the uniqueness of Whittaker model, we have $\dim (\pi_{N,\psi_N})_{N_2, \psi_{\alpha}} \le 1$ and by our assumption, we have $$\dim (\pi_{N,\psi_N})_{N_2, \psi} = \dim \pi_{U, \psi_{U}}=1 .$$
 Thus after excluding a finite number of $q^{-s}$, we have
\begin{align*}\Hom_{\widetilde P_2}(\pi_{N,\psi_N}\otimes \omega_{\psi^{-1}}, \tilde \eta_s)&=\oplus_{\alpha\in F^\times/F^{\times,2}}\Hom_{\widetilde P_2}(\ind_{N_2}^{P_2} (\psi_\alpha)\otimes \omega_{\psi^{-1}}, \widetilde \eta_s)\\
&=\oplus_{\alpha\in F^\times/F^{\times,2}}\Hom_{P_2}(\ind_{N_2}^{P_2}(\psi_\alpha), \omega_\psi\otimes \wt\eta_s^{-1})\\
&=\oplus_{\alpha\in F^\times/F^{\times,2}}\Hom_{N_2}(\psi_\alpha, \omega_\psi\otimes \wt \eta_s^{-1}).
\end{align*}
The result follows from the fact that $(\omega_\psi\otimes \wt\eta_s^{-1})_{N_2,\psi_\alpha}=0$ if $\alpha\ne 1$ and $\dim (\omega_\psi\otimes \wt\eta_s^{-1})_{N_2,\psi_\alpha}=1 $, which could be checked easily from the Weil representation formulas. 
\end{proof}

Let $M_s: I(s,\eta,\psi^{-1})\to I(1-s,\eta^{-1},\psi^{-1})$ be the standard intertwining operator defined by 
$$M_s(f_s)(g)=\int_{N_2}f_s(w_2ng)dn,$$
where recall that $w_2=\left( \begin{smallmatrix} &1\\ -1& \end{smallmatrix} \right)\in \SL_2(F)$. Consider the character $\chi$ of $F^\times$ defined by $\chi(a)=(a,-1)_F$.

\begin{cor}
There is a meromorphic function $\gamma(s,\pi, \wedge^2_+\times\chi\eta,\psi)$ such that
$$\Psi(W,\phi, M_s(f_s))=\gamma(s,\pi,\wedge^2_+\times\chi\eta)\Psi(W,\phi, f_s),$$
for all $W\in \CW(\pi,\psi), \phi\in \CS(F)$ and $f_s\in I(s,\eta,\psi)$.
\end{cor}
\begin{proof}
Using Lemma \ref{lem: conjugation}, we can check that both trilinear forms $(W,\phi,f_s)\mapsto \Psi(W,\phi,f_s)$ and $(W,\phi,f_s)\mapsto \Psi(W,\phi,M_s(f_s))$ satisfy the conditions in Proposition \ref{multiplicity one}. Thus by the uniqueness of such trilinear forms, these two trilinear forms are proportional. Denote this proportion by $\gamma(s,\pi,\wedge^2_+\times \chi \eta,\psi)$ temporarily. By Lemma \ref{lem: convergence}, $\gamma(s,\pi,\wedge^2_+\times \chi \eta,\psi)$ is a meromorphic function of $s$.
\end{proof}
The notation $\gamma(s,\pi,\wedge^2_+\times \chi\eta,\psi)$ will be justified by the unramified calculation in the next subsection.

\subsection{Unramified calculation} Let $\pi$ be an unramified representation of $\SO_4(F)$ with Satake parameter $\diag(a,b,b^{-1},a^{-1})\in \SO_4(\C)$. Let $W^0$ be the unramified Whittaker function for $\pi$ normalized so that $W^0(I_4)=1$.
Let $a_k=t(\varpi^k,\varpi^{-k})$. By the Casselman-Shalika formula \cite{CS}*{Theorem 5.4}, we have 
$$W^0(a_k)=\frac{q^{-k}}{(a-b)(ab-1)}(a^{k+2}b^{-k+1}-a^{-k+1}b^{k+2}-a^{k+1}b^{-k}+a^{-k}b^{k+1})$$
if $k\ge 0$, and $W^0(a_k)=0$ if $k<0$. For $h\in  \SL_2(F)=N_2 T_2 K_2$, with $K_2=\SL_2(\fo_F)$, we can write $h=n\diag(a,a^{-1})k$. The Haar measure on the quotient $N_2\backslash \SL_2(F)$ is $dh=|a|^{-2}dk da. $ We also assume that $\psi$ is unramified. Thus we have $\mu_{\psi^{-1}}(u)=1$ for $u\in \fo_F^\times$. Let $\phi^0\in \CS(F)$ be the characteristic function of $\fo_F$ and let $f_s^0\in I(s,\eta,\psi^{-1})$ be the function such that $f(k)=1$ for $k\in K_2$.

 We have $$f_s^0(\diag(\varpi^k, \varpi^{-k}))=\mu_{\psi^{-1}}(\varpi^k)\eta(\varpi)^k|\varpi^k|^{s+1/2},$$ 
 and 
 $$(\omega_{\psi^{-1}}(\diag(\varpi^k, \varpi^{-k}))\phi^0)(1)=|\varpi^k|^{1/2}\mu_{\psi^{-1}}(\varpi^k).$$ 
Notice that $$\mu_{\psi^{-1}}(\varpi^k)\mu_{\psi^{-1}}(\varpi^k)=\mu_{\psi^{-1}}(\varpi^{2k})(\varpi^k, \varpi^k)=(\varpi^k,\varpi^k)=(\varpi,-1)^k=\chi(\varpi)^k.$$
Thus, we obtain
\begin{align*}
\Psi(W^0,\phi^0,f_s^0)&=\sum_{k=0}^{\infty}W^0(a_k)\eta(\varpi)^k q^{-ks}\chi(\varpi)^k\\
&=\frac{1}{(a-b)(ab-1)}\sum_k (a^{k+2}b^{-k+1}-a^{-k+1}b^{k+2}-a^{k+1}b^{-k}+a^{-k}b^{k+1})t^k\\
&=\frac{1+t}{(1-ab^{-1}t)(1-a^{-1}bt)}\\
&=\frac{1-t^2}{(1-t)(1-ab^{-1}t)(1-a^{-1}bt)},
\end{align*}
where $t=\chi(\varpi)\eta(\varpi)q^{-s}.$ Hence we get
$$\Psi(W^0, \phi^0, f_s^0)=\frac{L(s,\pi, \wedge^2_+\times \chi\eta)}{L(2s,\eta^2)}.$$
We summarize the above calculation in the following 
\begin{prop}
Assume that $\pi$ is an unramified representation of $\SO_4(F)$ with Satake parameter $\diag(a,b,b^{-1},a^{-1})$, $W^0\in \CW(\pi,\psi)$ is the unramified Whittaker function such that $W^0(I_4)=1$, $\phi^0\in \CS(F)$ is the characteristic function of $\fo_F$ and $f_s^0\in I(s,\eta,\psi^{-1})$ is the function such that $f_s^0(k)=1$ for $k\in \SL_2(\fo_F)$. Then we have $$\Psi(W^0, \phi^0, f_s^0)=\frac{L(s,\pi, \wedge^2_+\times \chi\eta)}{L(2s,\eta^2)}.$$
\end{prop}

\subsection{A global integral}
In this subsection, we present a global integral so that it is Eulerian and its local piece at a finite local place is the the local zeta integral defined in $\S$\ref{subsection: local zeta integral}. In this subsection, $F$ is a global field such that its characteristic is not $2$. Let $\BA$ be the ring of adeles of $F$. Let $\psi$ be a fixed additive character of $F\backslash \BA$.

Let $\eta$ be a quasi-character of $F^\times \backslash \BA^\times$, and $s\in \C$. We consider the global induced representation 
$$I(s,\eta,\psi^{-1})=\Ind_{\widetilde P_2(\BA)}^{\widetilde \SL_2(\BA)}(\mu_{\psi^{-1}} \eta_{s-1/2}).$$
Here $\mu_{\psi^{-1}}$ is a global version of the local $\mu_{\psi^{-1}}$ defined in $\S$\ref{subsection: Weil}.  For a standard section $f_s\in I(s,\eta,\psi^{-1})$, we consider the Eisenstein series
$$E(s,h,f_s)=\sum_{\gamma \in B_2(F)\backslash \SL_2(F)}f_s(\gamma h ).$$
It is standard that $E(s,h,f_s)$ converges for $\Re(s)\gg 0$ and has a meromorphic extension to $\C$. There is also a global Weil representation $\omega_\psi$ of $\wt \SL_2(\BA)$ on $\CS(\BA)$, the space of Bruhat-Schwatz functions on $\BA$. For $\phi\in \CS(\BA)$, we consider the theta series on $\widetilde \SL_2(\BA)$:
$$\theta(\phi)(g)=\sum_{x\in F}(\omega_{\psi}(g)\phi)(x).$$
We have
$$\theta(\phi)(g)=\theta_0(\phi)(g)+\sum_{a\in F^\times}|a|^{-1/2}\mu_{\psi}(a)^{-1}(\omega_{\psi}(t_{\SL_2}(a)g)\phi)(1),$$
where $\theta_0(\phi)(g)=\omega_{\psi}(g)\phi(0)$ is the contribution from $x=0$ in the theta series.

Let $\varphi$ be a cusp form on $\SO_4(F)\backslash \SO_4(\BA)$, $f_s$ be a standard section in $I(s,\eta,\psi^{-1})$, $\theta(\phi)$ be a theta series on $\widetilde \SL_2(\BA)$ associated to $\phi\in \CS(\BA)$. We consider the integral
$$Z(s,\varphi, \theta(\phi), f_s)=\int_{\SL_2(F)\backslash \SL_2(\BA)} \int_{N(F)\backslash N(\BA)} \varphi(n\bm(g) )\psi'(n)dn \theta(\phi)(g)E(s,g,f_s)dg.$$
By the definition of $E$, we get
\begin{align*}
Z(s,\varphi,\theta(\phi),f_s)&=\int_{B_2(F)\backslash \SL_2(\BA)}\int_{N(F)\backslash N(\BA)}\varphi(n\bm(g))\psi'(n)dn \theta(\phi)(g)f_s(g)dg.
\end{align*}
Now plugging in the definition of $\theta(\phi)$, and note that the contribution from $\theta_0(\phi)$ is 
\begin{equation*}
\int_{N_2(\BA) T_2(F)\backslash \SL_2(\BA)}\int_{N_2(F)\backslash N_2(\BA)}\int_{N(F)\backslash N(\BA)} \varphi(n\bm(u) \bm(g))\psi'(n) \omega_\psi(g)\phi(0)   f_s(g)dndudg,
\end{equation*}
which equals zero because $\varphi$ is a cusp form. Thus after absorbing the sum over $a\in F^\times$, we have
\begin{align*}
&Z(s,\varphi,\theta(\phi),f_s)\\
&\qquad =\int_{N_2(F)\backslash \SL_2(\BA)}\int_{N(F)\backslash N(\BA)}\varphi(n\bm(g))\psi'(n)dn (\omega_\psi(g)\phi)(1)f_s(g)dg\\
&\qquad =\int_{N_2(\BA)\backslash \SL_2(\BA)} \int_{N_2(F)\backslash N_2(\BA)}\int_{N(F)\backslash N(\BA)} \varphi(n\bm(u)\bm(g))\psi'(n)dn (\omega_\psi(ug)\phi)(1)f_s(ug)du dg\\
&\qquad =\int_{N_2(\BA)\backslash \SL_2(\BA)} \int_{N_2(F)\backslash N_2(\BA)}\int_{N(F)\backslash N(\BA)} \varphi(n\bm(u)\bm(g))\psi'(n)dn \psi(u)du (\omega_\psi(g)\phi)(1)f_s(g) dg\\
&\qquad =\int_{N_2(\BA)\backslash \SL_2(\BA)}\int_{U(F)\backslash U(\BA)} \varphi(u\bm(g))\psi_U(u)du (\omega_\psi(g)\phi)(1)f_s(g) dg\\
&\qquad =\int_{N_2(\BA)\backslash \SL_2(\BA)} W_\varphi^{\psi_U^{-1}}(\bm(g)) (\omega_\psi(g)\phi)(1)f_s(g) dg.
\end{align*}
Thus the global integral is Eulerian and its local piece is exactly the local integral defined in \S\ref{subsection: local zeta integral}.


\section{Partial Bessel functions} \label{section-partial-bessel-functions}

Write $G=\SO_4(F)$ and $Z=\wpair{\pm I_4}$ the center of $G$. In this section, we assume that $\psi$ is an unramified character of $F$. For a character $\omega$ of $Z$, let $C_c^\infty(G,\omega)$ be the space of compactly supported smooth functions $f$ on $G$ such that $f(zg)=\omega(z)f(g)$ for all $z\in Z, g\in G$. Moreover, denote $C^\infty(G,\psi_U,\omega)$ the space of smooth functions $W$ on $G$ such that $W(zug)=\omega(z)\psi_U(u)W(g)$ for all $z\in Z, u\in U, g\in G$. Here the smoothness of $W$ means that for each $W$, there exists an open compact subgroup $K$ of $G$ such that $W(gk)=W(g)$ for all $g\in G, k\in K$.

\subsection{Howe vectors}
\label{subsection-howe-vectors}
We recall the theory of Howe vectors as developed in \cite{Ba95}. For a positive integer $m$, let $K_m=(1+\Mat_{4\times 4}(\fp^m))\cap \SO_4(F)$. Define a character $\tau_m$ of $K_m$ by 
$$\tau_m(k)=\psi(\varpi^{-2m}(k_{12}-2k_{13})).$$
One can check that $\tau_m$ is indeed a character of $K_m$. 

Let $d_m=t(\varpi^{-2m},1).$ Consider the subgroup $H_m=d_m K_m d_m^{-1}$. Define $\psi_m(h)=\tau_m(d_m^{-1}h d_m)$ for $h\in H_m$. Let $U_{m}=U\cap H_m$. We then have 
$$U_{m}=\wpair{u(x,y): x,y\in \fp^{-m}}$$  and $\psi_m|_{U_m}=\psi_U|_{U_m}$.

For a positive integer $m$ and an element $W\in C^\infty(G,\psi_U,\omega)$ with $W(1)=1$, following \cite{Ba95}, we consider 
\begin{equation}\label{eq: defn of Howe}W_m(g)=\frac{1}{\vol(U_m)}\int_{U_m}W(gu)\psi_m^{-1}(u)du. \end{equation}

For a fixed $W$, let $C$ be a positive integer such that $W$ is invariant under the right translation by $K_C$, then a function $W_m$ with $m\ge C$ is called a \textbf{Howe} vector. It is known that $W_m(1)=1$. In particular, $W_m\ne 0$. Moreover, for $m\ge C$, we have 
\begin{equation}\label{partial bessel}
W_m(ugh)=\psi_U(u)\psi_m(h)W_m(g), \forall u\in U, h\in H_m, g\in G.
\end{equation}
Because of \eqref{partial bessel}, the functions $W_m$, $m\ge C$, are called partial Bessel functions. For a proof of the statements in this subsection, see \cite{Ba95}.

\subsection{Partial Bessel function and Bruhat order}
\label{subsection-partial-bessel}
Let $\pi$ be an irreducible generic supercuspidal representation of $G$ with central character $\omega$. Let $\CM(\pi)$ be the space of matrix coefficients of $\pi$. Then we have $\CM(\pi)\subset C_c^\infty(G)$. For $f\in \CM(\pi)$, we consider 
$$W^f(g)=\int_U f(ug)\psi^{-1}_U(u)du.$$
Note that the above integral makes sense because $Ug$ is closed in $G$ and $f$ has compact support in $G$. We have $W^f\in C^\infty(G,\psi_U,\omega)$. Moreover, since $\pi$ is generic, there exists an $f\in \CM(\pi)$ such that $W^f(1)=1$. For a positive integer $m$, we can consider the function 
$$\CB_m(g,f):=(W^f)_m(g), g\in G.$$

For a Weyl element $w\in \bW$ of $G$, we denote $C(w)=BwB$. Recall the Bruhat order on $\bW$ is defined as $w_1\ge w_2$ if and only if $C(w)\subset \ov{C(w_1)}$. For $w\in \bW$, we consider the open set 
$$\Omega_w=\cup_{w'\ge w}C(w)$$
of $G$. For $w\in \bW$, we consider $A_w=\wpair{t\in T| \gamma(t)=1, \forall \gamma\in \Delta \textrm{ with } w\gamma>0}.$ Note that the Bruhat order is particularly simple for $\SO_4$. We recall the following result of \cite{CST}, specializing to our case when $G=\SO_4$.

\begin{lem}[{\cite{CST}*{Lemma 5.13}}] \label{CST}
Let $w\in \bW$, $m>0$ and $f\in C_c^\infty(\Omega_w, \omega)$. Suppose that $\CB_m(wa, f)=0$ for all $a\in A_w$. Then there exists a function $f_0\in C_c^\infty(\Omega_w-C(w),\omega)$ such that for sufficiently large $m $ depending only on $f$, we have $\CB_m(g,f)=\CB_m(g,f_0)$ for all $g\in G$.
\end{lem}

\subsection{Several preparation results}
In this subsection, we collect several preparation results which will be used in the proof of our local converse theorem.

Recall that $N_2=\{\bn(x)=\left(\begin{smallmatrix} 1 & x \\ & 1\end{smallmatrix}\right), x\in F \}$ is the upper triangular unipotent subgroup of $\SL_2(F)$. Let $\ov N_2 =\{\bar \bn(x) =\left(\begin{smallmatrix} 1 &  \\ x& 1\end{smallmatrix}\right), x\in F \}$ be the lower triangular unipotent subgroup of $\SL_2(F)$. Let $N_{2,m}=\begin{pmatrix} 1& \fp^{-m}\\ &1\end{pmatrix}$ and $\ov N_{2,m}=\begin{pmatrix}1&\\ \fp^{3m} &1 \end{pmatrix}$. 
 Note that $\ov N_2$ and $N_2$ splits in $\widetilde \SL_2(F)$. Moreover, for any $g_1\in N_2$ and $g_2\in \ov N_2$, we have $c(g_1,g_2)=1$. In fact, if $g_1=\bn(y)$ and $g_2=\bar \bn(x)$ with $x\ne 0$, we have $\bx(g_1)=1$ and $\bx(g_2)=x$, and thus
$$c(g_1,g_2)=( 1,x)_F(-x, x)_F=1.$$

For an integer $i$, we consider the following section $f_s^i\in I(s,\eta, \psi^{-1})$ defined by
$$f_s^i((g,\epsilon))=\left\{\begin{array}{lll}\epsilon \gamma_{\psi^{-1}}(a)\eta_{s+1/2}(a),& \textrm{ if } g=\bn(b)t_{\SL_2}(a) \bar \bn(x), \textrm{ with }  a\in F^\times, b\in F, \epsilon\in \wpair{\pm1}, x\in \fp^{3i}, \\ 0, &\textrm{ otherwise.} \end{array}\right.$$
Note that the support of $f_s^i$ is $\widetilde B_2 \ov N_{2,i}$, which is open in $\widetilde\SL_2(F)$. Thus $f_s^i$ is well-defined.

\begin{lem}[{\cite{CZ}*{Lemma 3.8}}]\label{lem: sections}
\begin{enumerate}
\item There exists an integer $i_2$ such that for all $i\ge i_2$, $f_s^i$ defines a section in $I(s,\eta,\psi^{-1})$.
\item Let $X$ be an open compact subset of $N_2$, then there exists an integer $I(X,\eta)$ such that for all $i\ge I(X,\eta)$, we have
$$\tilde f_s^i(w_2 x)=\vol(\ov N_{2,i})=q^{-3i}$$
for all $x\in X$, where $\tilde f_s^i=M_s(f_s^i)$.
\end{enumerate}
\end{lem}

Let $\phi^m$ be the characteristic function of $1+\fp^m$. We have the following
\begin{lem}[\cite{CZ}*{Lemma 3.9}] \label{lem: Weil representation} We have
\begin{enumerate} 
\item $\omega_{\psi^{-1}}(n)\phi^m=\psi^{-1}(n) \phi^m$ for all $n\in N_{2,m}$.
\item $\omega_{\psi^{-1}}(\bar n)\phi^m=\phi^m$ for all $\bar n\in \ov N_{2,m}.$
\item $\omega_{\psi^{-1}} (w_2)\phi^m (a)=\gamma(\psi^{-1})\psi^{-1}(2a)q^{-m}$.
\end{enumerate}
\end{lem} 

\subsection{A local converse theorem of Hazeltine-Liu and Haan-Kim-Kwon}\label{subsection-HL-Theorem}
In this subsection, we recall the local converse theorem of Hazeltine and Liu for the split group $\SO_{2n}(F)$, which is realized by $J_{2n}$, namely, $\SO_{2n}(F)=\wpair{g\in \SL_{2n}(F): g^t J_{2n}g=J_{2n}}$.  The outer automorphism $c$ of split $\SO_{2n}$ can be realized by the matrix 
$c=\diag( I_{n-1},J_2, I_{n-1}).$
Given an irreducible generic representation $\pi$ of $\SO_{2n}(F)$ and an irreducible generic representation $\tau$ of $\GL_k(F)$, one can associate a local gamma factor $\gamma(s,\pi\times \tau,\psi)$; see \cite{Kaplan}. In \cite{HL: p-adic} and \cite{HKK}, Hazeltine-Liu and Haan-Kim-Kwon proved the following local converse theorem for $\SO_{2n}(F)$ independently.
\begin{thm}[\cite{HL: p-adic, HKK}]\label{HL theorem}
Let $\pi,\pi'$ be two irreducible $\psi$-generic representations of $\SO_{2n}(F)$ with the same central character. If $\gamma(s, \pi\times \tau,\psi)=\gamma(s, \pi'\times \tau,\psi)$ for any irreducible generic representation $\tau$ of $\GL_k(F)$ with $1\le k\le n$, then either $\pi\cong \pi^\prime$ or $\pi\cong c\cdot \pi^\prime$.
\end{thm}

\section{On a refined local converse theorem for $\SO_4$} \label{section-proof}

In this section, we prove Theorem~\ref{converse theorem}.

\subsection{A preparation step} \label{subsection: key}
Let $\pi$ and $\pi'$ be two irreducible $\psi$-generic supercuspidal representations of $\SO_4(F)$ with the same central character $\omega$. Take $f\in \CM(\pi)$ (resp. $f'\in \CM(\pi')$) such that $W^f(1)=W^{f'}(1)=1$. Note that there are many such choices of $f$ and $f'$. We then can consider the Howe vectors $W^{f}_m$ (resp. $W^{f'}_m$) and partial Bessel functions $\CB_m(g,f)$ (resp. $\CB_m(g,f)$). We will show that $\CB_m(g,f)=\CB_m(g,f')$ under the assumption of Theorem \ref{converse theorem}.

We start from the following lemma, which is a direct application of Lemma \ref{CST}.
\begin{lem}\label{lem: first decomposition}
There exist functions $f_\alpha\in C_c^\infty( \Omega_{s_\alpha},\omega)$ and $f_\beta\in C_c^\infty( \Omega_{s_\beta},\omega)$ such that 
$$\CB_m(g,f)-\CB_m(g,f')=\CB_m(g,f_\alpha)+\CB_m(g,f_\beta)$$
for $m$ large enough, which only depend on $f,f'$.
\end{lem}
\begin{proof}
Recall that for $w\in \bW$, we have defined $A_w$ in Section~\ref{subsection-partial-bessel}.
Write the trivial element in $\bW$ by 1. Then $A_1$ is the center $Z$ of $\SO_4(F)$. By the assumption on the central character, we have $\CB_m(z,f-f')=0$ for all $z\in Z$ and $m$ large. Thus by Lemma \ref{CST}, there is a function $f_0\in C_c^\infty(\Omega_1-C(1),\omega)$ such that 
$$\CB_m(g,f-f')=\CB_m(g,f_0). $$
Note that $\Omega_1-C(1)=\Omega_{s_\alpha}\cup \Omega_{s_\beta}$. Thus a partition of unity argument will give functions $f_\alpha\in C_c^\infty(\Omega_{s_\alpha},\omega)$ and $f_\beta\in C_c^\infty( \Omega_{s_\beta},\omega)$ such that 
$$\CB_m(g,f-f')=\CB_m(g,f_0)=\CB_m(g,f_\alpha)+\CB_m(g,f_\beta), $$
for $m$ large. 
\end{proof}

\begin{lem}\label{lem: stable partial bessel function}
\begin{enumerate}
\item For $m$ large and for any $t\in T$ and $r\in F$, we have $$\CB_m(ts_\alpha \bx_\alpha(r), f_\beta)=0.$$
\item For $m$ large and for any $t\in T$ and $r\in F$ with $\bx_{\alpha}(r)\notin U_m$, we have 
$$\CB_m(ts_\alpha \bx_\alpha(r),f_\alpha)=0.$$
\end{enumerate}
\end{lem}
\begin{proof}
(1) Note that $ts_\alpha \bx_\alpha (r)\in Bs_\alpha B$ while $\Omega_{s_\beta}=Bs_\beta B\cup Bs_\alpha s_\beta B$. Thus $ ts_\alpha \bx_\alpha(r)\notin \Omega_{s_\beta}$. The result follows.

(2) Note that $Bs_\alpha B$ is closed in $\Omega_{s_\alpha}=Bs_\alpha B\cup Bs_\alpha s_\beta B$. Thus $\supp(f_\alpha)\cap Bs_\alpha B$ is compact. Note that the map 
$$B\times F\to Bs_\alpha B$$
$$(b,r)\to b s_\alpha \bx_\alpha(r)$$
is a homeomorphism. Thus there exists an open compact subset $B_c\subset B$ and $F_c\subset F$ such that if $f_\alpha(bs_\alpha \bx_{\alpha}(r))\ne 0$, then $b\in B_c$ and $r\in F_c$. We take $m$ large enough so that $\bx_\alpha(F_c)\subset U_m$. We remind the reader that $U_m$ was defined in Section~\ref{subsection-howe-vectors}. Note that this choice of $m$ only depends on $f_\alpha$. We then have 
\begin{align*}
\CB_m(ts_\alpha \bx_\alpha(r),f_{\alpha})&=\frac{1}{\vol(U_m)}\int_{U\times U_m}\psi_U^{-1}(uu')f_{\alpha}(uts_\alpha \bx_\alpha(r)u')dudu'.
\end{align*}
For $u'\in U_m$, we can write $u'=\bx_\alpha(s_1)\bx_\beta(s_2)$. Note that $s_\alpha \bx_\beta(s_2)=u''s_\alpha$ for some $u''\in U$. Thus
$$f_{\alpha}(uts_\alpha \bx_\alpha(r)u')=f_\alpha(utu'' s_\alpha \bx_{\alpha}(r+s_1)). $$
Note that if $\bx_\alpha(r)\notin U_m$, then $\bx_\alpha(r+s_1)\notin U_m$ for $\bx_\alpha(s_1)\in U_m$. Thus $f_{\alpha}(uts_\alpha \bx_\alpha(r)u')=0 $ for any $u'\in U_m$ under the assumption $\bx_\alpha(r)\notin U_m$. Thus $\CB_m(ts_\alpha \bx_\alpha(r),f_{\alpha})=0. $
\end{proof}

\begin{prop}\label{prop: main}
Assume that $\gamma(s,\pi,\wedge^2_+\times \eta,\psi)=\gamma(s,\pi',\wedge^2_+\times \eta,\psi)$ for every quasi-character $\eta$ of $F^\times$, then we have 
$$\CB_m(g, f)=\CB_m(g,f'), \forall g\in Bs_\alpha B,$$
for $m$ large depending on $f$ and $f'$.
\end{prop}
The proof given below is indeed similar to the proof of \cite{CZ}*{Theorem 3.10}. For completeness, we still give the full details. 
\begin{proof}
Write $W_m$ for the Whittaker function $(W^f)_m=\CB_m(\cdot, f)$ or $(W^{f'})_m=\CB_m(\cdot, f')$. For a quasi-character $\eta$ of $F^\times$, we consider $f_s^i\in I(s,\eta,\psi^{-1})$ defined in Lemma \ref{lem: sections} and $\phi^m\in \CS(F)$, the characteristic function of $1+\fp^m$ as in Lemma \ref{lem: Weil representation}. Then we consider the integral $\Psi(W_m,\phi^m,f_s^i)$. This integral could be computed on the dense subset $N_2\backslash N_2 T_2 \ov N_2\subset N_2\backslash \SL_2$. Notice that $\bm(\bar \bn(x))=\bx_{-\alpha}(x)$ and $\bm(t_{\SL_2}(a))=t(a,a^{-1})$. We have
\begin{align*}
&\Psi(W_m, \phi^m, f_s^i)\\
&\quad =\int_{F^\times \times F}W_m(t(a,a^{-1})\bx_{-\alpha}(x)) (\omega_{\psi^{-1}}(t_{\SL_2}(a)\bar\bn(x)))\phi^m(1) f_s^i(t_{\SL_2}(a)\bar \bn(x))|a|^{-2}dx da\\
 &\quad=\int_{F^\times \fp^{3i}} W_m(t(a,a^{-1})  \bx_{-\alpha}(x))(\omega_{\psi^{-1}}(t_{\SL_2}(a)\bar \bn(x)))\phi^m(1) \gamma_{\psi^{-1}}(a)\eta_{s+1/2}(a) |a|^{-2}dx da.
\end{align*}

For $i\ge m$, we have $\bx_{-\alpha}(x)\in H_m$. Thus by \eqref{partial bessel} and Lemma \ref{lem: Weil representation}, we have
$$W_m(t_{\SL_2}(a) \bar \bn(x))=W_m(t_{\SL_2}(a)),  (\omega_{\psi^{-1}}(t_{\SL_2}(a)\bar \bn(x)))\phi^m(1)=(\omega_{\psi^{-1}}(t_{\SL_2}(a)))\phi^m(1).$$
Since $ (\omega_{\psi^{-1}}(t_{\SL_2}(a)))\phi^m(1)=\mu_{\psi^{-1}}(a)|a|^{1/2} \phi^m(a)$ and $\mu_{\psi^{-1}}(a)\mu_{\psi^{-1}}(a)=(a,a)=(a,-1)=\chi(a)$, we get
\begin{align}
\label{eq-value-of-local-integral}
\Psi(W_m, \phi^m, f_s^i)&=q^{-3i} \int_{F^\times}W_m(t(a,a^{-1}))(\omega_{\psi^{-1}}(t_{\SL_2}(a)))\phi^m(1) \gamma_{\psi^{-1}}(a)\eta_{s+1/2}(a)|a|^{-2}da \\  \nonumber
&=q^{-3i}\int_{1+\fp^m}W_{m}(t(a,a^{-1}))\chi(a)\eta_{s-1}(a)da\\  \nonumber
&=q^{-3i-m}, 
\end{align}
if $m\ge \cond(\chi\eta)$. In particular, we have 
\begin{equation}\label{eq: left side}\Psi((W^f)_m,\phi^m, f_s^i)=\Psi((W^{f'})_m,\phi^m, f_s^i).\end{equation}
By the assumption on the local gamma factors and the local functional equation, \eqref{eq: left side} implies that
\begin{equation}\label{eq: right side}
\Psi((W^f)_m,\phi_m, \wt f_s^i)=\Psi((W^{f'})_m,\phi_m, \wt f_s^i).
\end{equation}
The integral $\Psi(W_m, \phi^m, \wt f_s^i)$ can be taken on the open dense subset $N_2\backslash N_2 T_2 w_2 N_2$ of $N_2\backslash \SL_2$. Notice that $\bm(w_2)=s_\alpha$ and $\bm(\bn(x))=\bx_{\alpha}(x)$. We have
\begin{align*}
&\Psi(W_m, \phi^m,\tilde f_s^i)\\
\quad&=\int_{F^\times \times F}W_m(t(a,a^{-1})s_\alpha \bx_\alpha(x))(\omega (t_{\SL_2}(a)w_2 \bn(x)))\phi^m(1) \tilde f_s^i(t_{\SL_2}(a)w_2 \bn(x))|a|^{-2} dx da\\
\quad&=\int_{F^\times \times \fp^{-m}} W_m(t(a,a^{-1})w_2 \bx_\alpha(x))(\omega_{\psi^{-1}} (t_{\SL_2}(a)w_2 \bn(x)))\phi^m(1) \tilde f_s^i(t_{\SL_2}(a)w_2 \bn(x))|a|^{-2} dx da\\
\quad &+\int_{F^\times \times (F-\fp^{-m})}W_m(t(a,a^{-1})s_\alpha \bx_{\alpha}(x))(\omega_{\psi^{-1}} (t_{\SL_2}(a)w_2 \bn(x)))\phi^m(1) \tilde f_s^i(t_{\SL_2}(a)w_2 \bn(x)) |a|^{-2}dx da.
\end{align*}
By Lemma \ref{lem: first decomposition} and Lemma \ref{lem: stable partial bessel function}, for $m$ large, we have 
$$(W^f)_m(t(a,a^{-1})s_\alpha\bx_\alpha(x))=(W^{f'})_m(t(a,a^{-1})s_\alpha\bx_\alpha(x)), \forall x\in F-\fp^m.$$
Thus for $i$ large, by Lemma \ref{lem: sections} and Lemma \ref{lem: Weil representation}, we get
\begin{align*}
&\Psi((W^f)_m, \phi^m, f_s^i)-\Psi((W^{f'})_m, \phi^m, f_s^i)\\
=&\int_{F^\times \times \fp^{-m}} \CB_m(t(a,a^{-1})s_\alpha \bn(x), f-f')(\omega_{\psi^{-1}} (t_{\SL_2}(a)w_2 \bn(x)))\phi^m(1) \tilde f_s^i(t_{\SL_2}(a)w_2 \bn(x)) dx da\\
=&q^{-3i+m} \int_{F^\times}\CB_m(t(a,a^{-1})s_\alpha),f-f') \omega_{\psi^{-1}} (w_2)\phi^m (a) \chi(a)\eta_{-s-1}(a)da\\
=&q^{-3i}\gamma(\psi^{-1})\int_{F^\times} \CB_m(t(a,a^{-1})s_\alpha),f-f') \psi^{-1}(2a) \chi(a)\eta_{-s-1}(a)da.
\end{align*}
Thus \eqref{eq: right side} implies that
$$ \int_{F^\times}\CB_m(t(a,a^{-1})s_\alpha),f-f') \psi^{-1}(2a) \chi(a)\eta_{-s-1}(a)da=0.$$
Note that this equation is true for any quasi-character $\eta$ of $F^\times$ by assumption. Thus by the inverse Mellin transform, we get that
$$\CB_m(t(a,a^{-1})s_\alpha),f-f') \psi^{-1}(2a) =0, \forall a\in F^\times. $$
Since $\psi^{-1}(2a)\ne 0$, we then get 
\begin{equation}
\CB_m(t(a,a^{-1})s_\alpha),f-f')=0, \forall a\in F^\times.
\end{equation}

Note that $A_{s_\alpha}=\wpair{t(a,a^{-1}),a\in F^\times}$. Lemma \ref{CST} implies that $\CB_m(ts_\alpha,f)=\CB_m(ts_\alpha,f')$ for all $t\in T$ and for possibly larger $m$. Finally, notice that any element in $Bs_\alpha B$ can be written as the form $ts_\alpha \bx_\alpha(r)$ for some $r\in F$. This finishes the proof.
\end{proof}

\subsection{The group $\GSO_4$}\label{subsection: GSO4}

First, we introduce a closely related group, the orthogonal similitude group $\mathrm{GO}_4(F)$, which is realized as $\mathrm{GO}_4(F)=\{g\in \GL_4(F): g^t J_4 g =\lambda(g) J_4, \, \lambda(g)\in F^\times\}$.
Since $\det(g)^2=\lambda(g)^{4}$, it has two connected components depending on whether $\det(g)/\lambda(g)^2$ is $1$ or $-1$. We denote the identity component by $\GSO_4(F)$, i.e., $\GSO_4(F)=\{g\in \GL_4(F): g^t J_4 g =\lambda(g) J_4, \, \lambda(g)\in F^\times, \det(g)=\lambda(g)^2\}$.
Note that $\SO_4(F)=\{g\in \GSO_4(F): \lambda(g)=1\}$. Note that $c\in \RG\RO_4(F)-\GSO_4(F)$ and thus the conjugation by $c$ is an outer automorphism on $\GSO_4(F)$.

The two simple roots also define embeddings $\iota_{\alpha}, \iota_{\beta}:\GL_2(F)\to \GSO_4(F)$ which can be described explicitly as
\begin{equation*}
    \iota_{\alpha}\left( \begin{pmatrix}
     a & b 
     \\ c & d
    \end{pmatrix} \right)=\begin{pmatrix}
 a&b &0 & 0 \\
  c&d &0 & 0\\
   0&0 &a & -b\\
    0&0 &-c & d
\end{pmatrix}, \, \, 
\iota_{\beta}\left( \begin{pmatrix}
     a & b 
     \\ c & d
    \end{pmatrix} \right)=\begin{pmatrix}
 a&0 &b & 0 \\
  0&a &0 & -b\\
   c&0 &d & 0\\
    0&-c &0 & d
\end{pmatrix}.
\end{equation*}
Notice that $\lambda(\iota_\alpha(g))=\det(g)$ and $\lambda(\iota_\beta(g))=\det(g)$ for $g\in \GL_2(F)$. The embeddings $\iota_\alpha$ and $\iota_\beta$ define a homomorphism 
$$\GL_2(F)\times \GL_2(F)\to \GSO_4(F)$$
$$(g,h)\mapsto \iota_\alpha(g)\iota_\beta(h),$$
which gives an isomorphism 
\begin{equation*}
    \GSO_4(F)\cong (\GL_2(F)\times \GL_2(F) )/ \Delta(F^\times),
\end{equation*}
where $\Delta(F^\times)=\{ (a I_2, a^{-1} I_2): a\in F^\times\}\subset \GL_2(F)\times \GL_2(F)$. Thus an irreducible representation of $\GSO_4(F)$ is of the form $\pi_1\boxtimes \pi_2$ with $\omega_{\pi_1}=\omega_{\pi_2}$, where $\pi_i$ is an irreducible representation of $\GL_2(F)$ and $\omega_{\pi_i}$ is the central character of $\pi_i$.

Note that the outer automorphism $c$ on $\GSO_4(F)$ satisfies 
$$c\iota_\alpha(g)c=\iota_\beta(g), \forall g\in \GL_2(F). $$
Thus for an irreducible representation $\Pi=\pi_1\boxtimes \pi_2$ of $\GSO_4(F)$ with irreducible representations $\pi_1,\pi_2$ of $\GL_2(F)$ with $\omega_{\pi_1}=\omega_{\pi_2}$, we have
$$c\cdot \Pi=\pi_2\boxtimes \pi_1.$$

\subsection{Proof of Theorem \ref{converse theorem}}
\begin{proof}[Proof of Theorem \ref{converse theorem}]
By assumption and Theorem \ref{HL theorem}, either $\pi\cong \pi'$ or $\pi\cong c\cdot \pi'$. If $\pi\cong \pi'$, we are done. From now on, we assume that $\pi\cong c\cdot \pi'$. From this condition, we get that $c\cdot \pi\cong \pi'$. We will show that $\pi \cong c\cdot \pi \cong \pi'$ under the assumption of Theorem \ref{converse theorem}.

We fix the notations as in $\S$\ref{subsection: key}. By  the assumption $\gamma(s,\pi,\wedge^2_+\times \eta)=\gamma(s,\pi',\wedge^2_+\times \eta)$ for all quasi-character $\eta$ of $F^\times$ and Proposition \ref{prop: main}, we have 
\begin{equation}\label{eq: key}(W^f)_m(ts_\alpha)=(W^{f'})_m(ts_\alpha), \forall t\in T.\end{equation}

By \cite{GK}*{\S 2}, there exists an irreducible representation $\Pi$ of $\GSO_4(F)$ such that $\Pi|_{\SO_4(F)}$ contains $\pi$ as a direct summand. By the uniqueness of Whittaker functional of $\Pi$, we know that $\pi$ is the unique $\psi_U$-generic irreducible direct summand of $\Pi|_{\SO_4(F)}$.

We can write $\Pi=\pi_1\boxtimes \pi_2$ for irreducible representations $\pi_1,\pi_2$ of $\GL_2(F)$ with $\omega_{\pi_1}=\omega_{\pi_2}$. Note that $\GSO_4(F)$ and $\SO_4(F)$ share the same maximal unipotent subgroup. Assume that $\Lambda\in \Hom_U(\pi,\psi_U)$ is the unique $\psi_U$ Whittaker functional (up to scaler) of $\pi$. Then the map $\Pi\to \pi\to \psi_U$ gives a nonzero Whittaker functional of $\Pi$. By abuse of notations, we also denote this Whittaker functional by $\Lambda$. For $v\in \pi\subset \Pi$, we can consider $\wt W_v(g)=\Lambda(\Pi(g)v)$. Note that $\wt W_v|_{\pi}$ gives a Whittaker function of $\pi$ and any Whittaker function of $\pi$ is of this form. We write $\wt W^f$ the corresponding Whittaker function on $\GSO_4(F)$ such that $\wt W^f|_{\SO_4(F)}=W^f$. For a positive integer $m$, if we define 
\begin{equation}\label{eq: def of Howe}\wt W^f_m(g)=\frac{1}{\vol(U_m)}\int_{U_m}\wt W^f(gu)\psi_U^{-1}(u)du, g\in \GSO_4(F),\end{equation}
then $\wt W^f_m|_{\SO_4(F)}=W^f_m.$

On the other hand, by the decomposition $\Pi=\pi_1\boxtimes \pi_2$ and uniqueness of Whittaker functionals, the Whittaker functional $\Lambda$ on $\Pi$ has the form $\Lambda=\lambda_1\boxtimes \lambda_2$ for $\lambda_i\in \Hom_{U_i}(\pi_i, \psi_U|_{U_i})$, where $U_1=\wpair{\bx_\alpha(r): r\in F}$ and $U_2=\wpair{\bx_\beta(r):r\in F}$ are the corresponding maximal unipotent subgroups in $\iota_\alpha(\GL_2(F))$ and $\iota_\beta(\GL_2(F))$. Then $\CW(\Pi,\psi_U)$ is spanned by pure tensors of the form $W^1\boxtimes W^2$ with $W^i\in \CW(\pi_i,\psi_U|_{U_i}),$ where $(W^1\boxtimes W^2)(\iota_\alpha(g)\iota_\beta(h))=W^1(g)W^2(h)$. 

 We now assume that $f\in \CM(\pi)$ is chosen to be a pure tensor $f=f_1\boxtimes f_2$ for some $f_i\in \CM(\pi_i)$, namely, $$f(\iota_\alpha(g)\iota_\beta(h))=f_1(g)f_2(h).$$  Of course, we also assume that $W^f(1)=1$. Then we have  $\wt W^f =W^1\boxtimes W^2$ for $W^i(g)=\int_{U_i}f_i(ug)\psi_U^{-1}(u)du$, which is a Whittaker function of $\pi_i$. Using \eqref{eq: def of Howe}, we see that $\wt W^f_m=W^1_m\boxtimes W^2_m$, where $W^i_m(g):=\frac{1}{\vol(U_i\cap H_m)}\int_{U_i\cap H_m}W^i_m(gu)\psi_U^{-1}(u)du$, is a Howe vector for $\pi_i$.


By assumption, we have $\pi'\cong c\cdot \pi$ and thus $\Pi':=c\cdot \Pi=\pi_2\boxtimes \pi_1$ is an irreducible representation of $\GSO_4(F)$ such that $\Pi'|_{\SO_4(F)}$ contains $ \pi'$ as the unique $\psi_U$-generic direct summand.  Note that $f':=f_2\boxtimes f_1\in \CM(\Pi')$ satisfies $ W^{f'}(1)=1$. Moreover, by \eqref{eq: defn of Howe} and \eqref{eq: def of Howe}, we have $\wt W^{f'}_m|_{\SO_4(F)}=(W^{f'})_m$ and $\wt W^{f'}_m=W^2_m\boxtimes W^1_m$.

By \eqref{eq: key}, we have $\wt W^f_m(g)=\wt W^{f'}_m(g)$, for $g\in Bs_\alpha B$ and $m$ large enough. Notice that $Bs_\alpha B\subset \iota_\alpha (\GL_2)$. The above discussion shows that 
\begin{equation}\label{eq: GL2 equation}W_m^1(g)=W_m^2(g), \forall g\in B_{\GL_2}s_\alpha B_{\GL_2},\end{equation}
where $B_{\GL_2}$ is the upper triangular subgroup of $\GL_2(F)$. Since $\omega_{\pi_1}=\omega_{\pi_2}$, we have $W_{m}^1(z)=W_m^2(z)$ for all $z\in Z_{\GL_2}$. A simple application of an analogue of Lemma \ref{CST} in the $\GL_2$-case (see \cite{CST}*{Lemma 5.13}) shows that $W_m^1(g)=W_m^2(g)$ for all $g\in B_{\GL_2}$. By Bruhat decomposition and \eqref{eq: GL2 equation}, we get 
$$W_m^1(g)=W_m^2(g), \forall g\in \GL_2(F), $$
when $m$ is large enough. By uniqueness of Whittaker functional, we get that $\pi_1\cong \pi_2$. This implies that $\Pi'\cong c\cdot \Pi\cong \Pi$. Since $\pi$ (resp. $\pi'$) is the unique $\psi_U$-generic direct summand of $\Pi$ (resp. $\Pi'$), we get that $\pi\cong \pi'$. This concludes the proof.
\end{proof}


\begin{bibdiv}
\begin{biblist}

\bib{Arthur}{book}{ AUTHOR = {Arthur, James},
     TITLE = {The endoscopic classification of representations},
    SERIES = {American Mathematical Society Colloquium Publications},
    VOLUME = {61},
      NOTE = {Orthogonal and symplectic groups},
 PUBLISHER = {American Mathematical Society, Providence, RI},
      YEAR = {2013},
     PAGES = {xviii+590},
      ISBN = {978-0-8218-4990-3},
   MRCLASS = {22E55 (11F66 11F70 11F72 11R37 20G25 22E50)},
  MRNUMBER = {3135650},}
  
  \bib{Atobe}{article}{
  Author={Atobe, Hiraku}
  Title={On the Uniqueness of Generic Representations in an $L$-packet},
  Journal={International Mathematics Research Notices,},
  Volume={Vol. 2017}
  Number={23}
  Year={2017}
  Pages={7051-7068}
}


\bib{BAZ}{incollection}{
    AUTHOR = {Ben-Artzi, Asher}
    AUTHOR= {Soudry, David},
     TITLE = {{$L$}-functions for {${\rm U}_m\times R_{E/F}{\rm GL}_n\
              (n\leq[{m\over2}])$}},
 BOOKTITLE = {Automorphic forms and {$L$}-functions {I}. {G}lobal aspects},
    SERIES = {Contemp. Math.},
    VOLUME = {488},
     PAGES = {13--59},
 PUBLISHER = {Amer. Math. Soc., Providence, RI},
      YEAR = {2009},
      ISBN = {978-0-8218-4706-0},
   MRCLASS = {11F66 (11F70)},
  MRNUMBER = {2522026},
MRREVIEWER = {Anton\ Deitmar},
       DOI = {10.1090/conm/488/09563},
       URL = {https://doi.org/10.1090/conm/488/09563},
}


\bib{Ba95}{article}{
AUTHOR = {Baruch, Ehud Moshe},
     TITLE = {Local factors attached to representations of p-adic groups and
              strong multiplicity one},
      NOTE = {Thesis (Ph.D.)--Yale University},
 PUBLISHER = {ProQuest LLC, Ann Arbor, MI},
      YEAR = {1995},
     PAGES = {83},
   MRCLASS = {Thesis},
}

\bib{BZ}{article}{
AUTHOR = {Bern\v{s}te\u{\i}n, I. N.}
Author={Zelevinski\u{\i}, A. V.},
     TITLE = {Representations of the group {$GL(n,F),$} where {$F$} is a
              local non-{A}rchimedean field},
   JOURNAL = {Uspehi Mat. Nauk},
  FJOURNAL = {Akademiya Nauk SSSR i Moskovskoe Matematicheskoe Obshchestvo.
              Uspekhi Matematicheskikh Nauk},
    VOLUME = {31},
      YEAR = {1976},
    NUMBER = {3(189)},
     PAGES = {5--70},
      ISSN = {0042-1316},
   MRCLASS = {22E50},
  MRNUMBER = {0425030},}

\bib{CS}{article}{
AUTHOR = {Casselman, W.}
Author={ Shalika, J.},
     TITLE = {The unramified principal series of {$p$}-adic groups. {II}.
              {T}he {W}hittaker function},
   JOURNAL = {Compositio Math.},
  FJOURNAL = {Compositio Mathematica},
    VOLUME = {41},
      YEAR = {1980},
    NUMBER = {2},
     PAGES = {207--231},
      ISSN = {0010-437X},
   MRCLASS = {22E50},
  MRNUMBER = {581582},
MRREVIEWER = {J. Szmidt},}

\bib{CW}{article}{
Author={Cheng, Yao},
Author={Wang, Chian-Jen},
title={On gamma factors of generic representations of $\RU_{2n+1}\times \Res_{E/F}(\GL_r)$},
year={2023},
note={preprint, \href{https://arxiv.org/abs/2311.15323}{arXiv: 2311.15323}}
}

\bib{CST}{article}{
Author={Cogdell, J.}
Author={Shahidi, F.}
Author={Tsai, T-L.}
Title={ Local Langlands correspondence for $GL_n$ and the exterior and
symmetric square $\varepsilon$-factors}
Journal={Duke Math. J.}
Volume={166}
Year={2017}
Pages={2053-2132}
}

\bib{Chai}{article}{
AUTHOR = {Chai, Jingsong},
     TITLE = {Bessel functions and local converse conjecture of {J}acquet},
   JOURNAL = {J. Eur. Math. Soc. (JEMS)},
  FJOURNAL = {Journal of the European Mathematical Society (JEMS)},
    VOLUME = {21},
      YEAR = {2019},
    NUMBER = {6},
     PAGES = {1703--1728},
      ISSN = {1435-9855},
   MRCLASS = {11F70 (22E50)},
  MRNUMBER = {3945739},}

\bib{CZ}{article}{
    AUTHOR = {Chai, Jingsong}
    Author={Zhang, Qing},
     TITLE = {A strong multiplicity one theorem for {$\rm SL_2$}},
   JOURNAL = {Pacific J. Math.},
  FJOURNAL = {Pacific Journal of Mathematics},
    VOLUME = {285},
      YEAR = {2016},
    NUMBER = {2},
     PAGES = {345--374},
      ISSN = {0030-8730},
   MRCLASS = {11F70 (22E50 22E55)},
  MRNUMBER = {3575571},
MRREVIEWER = {Anton Deitmar},
       DOI = {10.2140/pjm.2016.285.345},
       URL = {https://doi.org/10.2140/pjm.2016.285.345},
}


\bib{FH}{book} {
    AUTHOR = {Fulton, William}
    Author={ Harris, Joe},
     TITLE = {Representation theory, a first course},
    SERIES = {Graduate Texts in Mathematics},
    VOLUME = {129},
 PUBLISHER = {Springer-Verlag, New York},
      YEAR = {1991},
     PAGES = {xvi+551},
      ISBN = {0-387-97527-6; 0-387-97495-4},
   MRCLASS = {20G05 (17B10 20G20 22E46)},
  MRNUMBER = {1153249},
MRREVIEWER = {James E. Humphreys},
       DOI = {10.1007/978-1-4612-0979-9},
       URL = {https://doi.org/10.1007/978-1-4612-0979-9},
}





\bib{GK}{article}{
    AUTHOR = {Gelbart, S. S.},
    AUTHOR = {Knapp, A. W.},    
     TITLE = {{$L$}-indistinguishability and {$R$} groups for the special
              linear group},
   JOURNAL = {Adv. in Math.},
  FJOURNAL = {Advances in Mathematics},
    VOLUME = {43},
      YEAR = {1982},
    NUMBER = {2},
     PAGES = {101--121},
      ISSN = {0001-8708},}


     
    \bib{GPSR}{book}{
    AUTHOR = {Gelbart, Stephen}
    Author={ Piatetski-Shapiro, Ilya}
    Author={ Rallis, Stephen},
     TITLE = {Explicit constructions of automorphic {$L$}-functions},
    SERIES = {Lecture Notes in Mathematics},
    VOLUME = {1254},
 PUBLISHER = {Springer-Verlag, Berlin},
      YEAR = {1987},
     PAGES = {vi+152},
      ISBN = {3-540-17848-1},
   MRCLASS = {11F70 (11R39 11S37 22E55)},
  MRNUMBER = {892097},
MRREVIEWER = {Joe Repka},
       DOI = {10.1007/BFb0078125},
       URL = {https://doi.org/10.1007/BFb0078125},
}


\bib{GPS80}{article}{
AUTHOR = {Gelbart, Stephen}
Author={Piatetski-Shapiro, I. I.}
     TITLE = {Distinguished representations and modular forms of
              half-integral weight},
   JOURNAL = {Invent. Math.},
  FJOURNAL = {Inventiones Mathematicae},
    VOLUME = {59},
      YEAR = {1980},
    NUMBER = {2},
     PAGES = {145--188},
      ISSN = {0020-9910},
   MRCLASS = {10D40 (22E55)},
  MRNUMBER = {577359},
MRREVIEWER = {Martin L. Karel},
       DOI = {10.1007/BF01390042},
       URL = {https://doi.org/10.1007/BF01390042},}

\bib{HKK}{article}
{author={Haan, Jaeho}
author={Kim, Yeansu}
author={Kwon, Sanghoon}
title={A local converse theorem for quasi-split $O_{2n}$ and $SO_{2n}$: the generic case}
year={2023}
journal={arXiv:2301.12693}
}

\bib{HL: finite}{article}
{author={Hazeltine, Alex}
author={Liu, Baiying}
title={A converse theorem for split $\SO_{2l}$ over finite fields},
year={2023}
journal={Acta Mathematica Sinica, English Series}
pages={to appear}
}

\bib{HL: p-adic}{article}
{author={Hazeltine, Alex}
author={Liu, Baiying}
title={On the local converse theorem for split $\SO_{2n}$},
year={2023}
note={preprint}
journal={arXiv:2301.13847}
}

\bib{JL}{article}{
 AUTHOR = {Jacquet, Herv\'{e}}
 author={ Liu, Baiying},
     TITLE = {On the local converse theorem for {$p$}-adic {${\rm GL}_n$}},
   JOURNAL = {Amer. J. Math.},
  FJOURNAL = {American Journal of Mathematics},
    VOLUME = {140},
      YEAR = {2018},
    NUMBER = {5},
     PAGES = {1399--1422},
      ISSN = {0002-9327},
   MRCLASS = {11S70 (11F85 22E57)},}
   
   \bib{JS}{article}{
   AUTHOR = {Jiang, Dihua}
   author={ Soudry, David},
     TITLE = {The local converse theorem for {${\rm SO}(2n+1)$} and
              applications},
   JOURNAL = {Ann. of Math. (2)},
  FJOURNAL = {Annals of Mathematics. Second Series},
    VOLUME = {157},
      YEAR = {2003},
    NUMBER = {3},
     PAGES = {743--806},}

\bib{Jo}{article}{
author={Jo, Yeongseong}
title={The local converse theorem for odd special orthogonal and symplectic groups in positive characteristic}
journal={arXiv:2205.09004}
year={2022}}

\bib{Kaplan}{article}{AUTHOR = {Kaplan, Eyal},
     TITLE = {Complementary results on the {R}ankin-{S}elberg gamma factors
              of classical groups},
   JOURNAL = {J. Number Theory},
  FJOURNAL = {Journal of Number Theory},
    VOLUME = {146},
      YEAR = {2015},
     PAGES = {390--447},}


      \bib{LZ}{article}
      {Author={Liu, Baiying}
      Author={Zhang, Qing}
      Title={On a converse theorem for $G_2$ over finite fields},
      Journal={Math. Ann.}
      Volume={383},
      Year={2022},
      Pages={1217-1283}
      }

\bib{Morimoto}{article}{
AUTHOR = {Morimoto, Kazuki},
     TITLE = {On the irreducibility of global descents for even unitary
              groups and its applications},
   JOURNAL = {Trans. Amer. Math. Soc.},
  FJOURNAL = {Transactions of the American Mathematical Society},
    VOLUME = {370},
      YEAR = {2018},
    NUMBER = {9},
     PAGES = {6245--6295},
      ISSN = {0002-9947},
   MRCLASS = {11F70 (11F30)},}
   
   \bib{Morimoto-gamma}{article}{
   author={Morimoto, Kazuki},
   title={On gamma factors of Rankin-Selberg integrals for $\RU_{2\ell}\times \Res_{E/F}(\GL_n)$},
   year={2023},
   note={preprint, \href{https://arxiv.org/abs/2306.07026}{arXiv: 2306.07026} },
   }
   
   \bib{YZ}{article}{
   Author={Yan, Pan},
   Author={Zhang, Qing},
   Title={Product of Rankin-Selberg convolutions and a new proof of Jacquet's local converse conjecture}
   notes={preprint,}
   journal={arXiv: 2309.10445}
   year={2023}
   }
   
   \bib{U11}{article}{
    AUTHOR = {Zhang, Qing},
     TITLE = {A local converse theorem for {$\rm U(1,1)$}},
   JOURNAL = {Int. J. Number Theory},
  FJOURNAL = {International Journal of Number Theory},
    VOLUME = {13},
      YEAR = {2017},
    NUMBER = {8},
     PAGES = {1931--1981},
      ISSN = {1793-0421},
   MRCLASS = {22E50 (11F70)},
  MRNUMBER = {3681684},}


\bib{Sp(2r)}{article}
{AUTHOR = {Zhang, Qing},
     TITLE = {A local converse theorem for {${\rm Sp}_{2r}$}},
   JOURNAL = {Math. Ann.},
  FJOURNAL = {Mathematische Annalen},
    VOLUME = {372},
      YEAR = {2018},
    NUMBER = {1-2},
     PAGES = {451--488},
      ISSN = {0025-5831},}
      
 \bib{U(2r+1)}{article}{
      AUTHOR = {Zhang, Qing},
     TITLE = {A local converse theorem for {${{\RU}}_{2r+1}$}},
   JOURNAL = {Trans. Amer. Math. Soc.},
  FJOURNAL = {Transactions of the American Mathematical Society},
    VOLUME = {371},
      YEAR = {2019},
    NUMBER = {8},
     PAGES = {5631--5654},
      ISSN = {0002-9947},}


\end{biblist}
\end{bibdiv}
\end{document}